\documentclass[a4paper,12pt]{amsart}
\usepackage{amssymb}
\usepackage{mathrsfs}
\usepackage{amsmath,amssymb,amsthm,latexsym,amscd,mathrsfs}
\usepackage{indentfirst}
\usepackage{stmaryrd}
\usepackage{graphicx}
\usepackage{extarrows}
\usepackage{amssymb}
\usepackage{amsmath,amssymb,amsthm,latexsym,amscd,mathrsfs}
\usepackage{indentfirst}
\usepackage{multirow}
\usepackage{latexsym}
\usepackage{amsfonts}
\usepackage{color}
\usepackage{pictexwd,dcpic}
\usepackage{graphicx}
\usepackage{psfrag}
\usepackage{hyperref}
\usepackage{comment}
\usepackage{cases}

\DeclareMathOperator{\oi}{\mathbf{i}}
\DeclareMathOperator{\oj}{\mathbf{j}}
\DeclareMathOperator{\ok}{\mathbf{k}}

 \DeclareMathOperator{\RE}{Re}

 \newcommand{\ROM}[1]{\mathrm{\uppercase\expandafter{\romannumeral#1}}}
  \theoremstyle{definition}
   \numberwithin{equation}{section} \theoremstyle{plain}

 \newtheorem{thm}{Theorem}[section]

\numberwithin{equation}{section} \theoremstyle{plain}
\newtheorem{prop}[thm]{Proposition}
\newtheorem{lem}[thm]{Lemma}

\newtheorem{rem}[thm]{Remark}
 \makeatletter

\newtheorem{ack}{Acknowledgements}   
  \makeatletter

  \numberwithin{equation}{section}
  \allowdisplaybreaks
\makeatother
\title[Extrinsic geometry of Gromoll-Meyer sphere]{\textbf{Extrinsic geometry of the Gromoll-Meyer sphere}}
\author[Chao Qian]{Chao Qian}\address{School of Mathematics and Statistics, Beijing Institute of Technology, Beijing
100081, P.R. China}
\email{6120150035@bit.edu.cn}
\author[Z. Z. Tang]{Zizhou Tang}\address{Chern Institute of Mathematics, Nankai University, Tianjin 300071, P. R. China}
\email{zztang@nankai.edu.cn}
\author[W. J. Yan]{Wenjiao Yan$^{\dag}$}\address{School of Mathematical Sciences, Laboratory of Mathematics and Complex Systems, Beijing Normal University, Beijing, 100875, P. R. China}
\email{wjyan@bnu.edu.cn}
\thanks {$^{\dag}$ the corresponding author}
\thanks {The project is partially supported by the NSFC (No.11722101, 11871282, 11931007), Nankai Zhide foundation, the Fundamental Research Funds for the Central Universities and and Beijing Institute of Technology Research Fund Program for Young Scholars.	}


\begin{document}

\maketitle

\begin{abstract}
Among a family of $2$-parameter left invariant metrics on $Sp(2)$, we determine which have nonnegative sectional curvatures and which are Einstein.
On the quotient $\widetilde{N}^{11}=(Sp(2)\times S^4)/S^3$, we construct a homogeneous isoparametric foliation with isoparametric hypersurfaces diffeomorphic to $Sp(2)$.
Furthermore, on the quotient $\widetilde{N}^{11}/S^3$, we construct a transnormal system with transnormal hypersurfaces diffeomorphic to the Gromoll-Meyer sphere $\Sigma^7$. Moreover,
the induced metric on each hypersurface has positive Ricci curvature and quasi-positive sectional curvature simultaneously.

\end{abstract}
\section{Introduction}
Since Milnor's discovery \cite{Mil56} of exotic $7$-spheres in 1956, the study of the exotic sphere has been one of the most intriguing problems in topology and Riemannian geometry. Exotic spheres are smooth manifolds which are homeomorphic but not diffeomorphic to standard spheres.
In 1974, Gromoll-Meyer \cite{GM74} produced the first example of an exotic sphere with a metric of non-negative sectional curvature, the so-called Gromoll-Meyer sphere. Recall the biquotient construction of the Gromoll-Meyer sphere as follows.

Let $Sp(2):=\Big\{Q=\Big(\begin{array}{cc}
          a& b\\
          c& d
        \end{array}\Big)\in M(2,\mathbb{H})|~QQ^*=I \Big\}$,
where $Q^*$ is the conjugate transpose of $Q$, and $\mathbb{H}$ is the algebra of quaternions which is identified with $\mathbb{R}^4$.
The Gromoll-Meyer sphere is obtained as the quotient of the following free action of $S^3=Sp(1)$ on $Sp(2)$:
\begin{eqnarray}\label{S3 action on sp2}
\phi_0:  S^3\times Sp(2) &\longrightarrow& \quad Sp(2) \\
  p,\quad Q \quad&\longmapsto& \Big(\begin{array}{cc}
          p& \\
          & p
        \end{array}\Big)Q\Big(\begin{array}{cc}
          \overline{p}& \\
          & 1
        \end{array}\Big),\nonumber
\end{eqnarray}
where $\overline{p}$ denotes the conjugate of $p$. In other words, $\Sigma^7=Sp(2)/\sim_{\phi_0}$ via regarding $Q\sim_{\phi_0} \big(\begin{array}{cc}
          p& \\
          & p
        \end{array}\big)Q\big(\begin{array}{cc}
          \overline{p}& \\
          & 1
        \end{array}\big)$.
If we denote by $\pi_0: Sp(2)\rightarrow \Sigma^7$ the projection induced by $\phi_0$, then $\pi_0$ is a submersion. For the bi-invariant metric on $Sp(2)$, it is clear that $\phi_0$ is an isometric action and the induced metric on $\Sigma^7$ has nonnegative sectional curvature.

Inspired by the construction of the Gromoll-Meyer sphere, the present paper studies the extrinsic geometry of $Sp(2)$ and $\Sigma^7$ by realizing them as isoparametric hypersurfaces and transnormal hypersurfaces in certain ambient spaces, respectively.

The isoparametric theory initiated from the study of E. Cartan in real space forms in 1940s, which caught much caution recently as the accomplishment of the classifications of isoparametric hypersurfaces in the unit sphere (\cite{CCJ07}, \cite{Chi13}, \cite{Chi16}, \cite{DN85}, \cite{Miy13}, \cite{Miy16}).
The applications of isoparametric theory are more and more abundant, see for example \cite{GTY20}, \cite{QTY13}, \cite{TXY14}, \cite{TY13} and \cite{TY15}.

By definition, a function $F: N\rightarrow \mathbb{R}$ on a Riemannian manifold $N$ is called \emph{transnormal} if there exists a smooth function $b$ on $\mathbb{R}$ such that $|\nabla^NF|^2=b(F)$, and \emph{isoparametric} if in addition there exists another continuous function $a$ on $\mathbb{R}$ such that $\Delta^N F=a(F)$, where $\nabla^N$ and $\Delta^N$ denote the gradient and Laplacian on $N$, respectively. The corresponding regular level sets of a transnormal function (respectively an isoparametric function) are called transnormal hypersurfaces (respectively isoparametric hypersurfaces) in $N$. The singular level sets are proved to admit manifold structures by \cite{Wan87}, and called focal submanifolds of the transnormal (isoparametric) hypersurfaces in $N$.

First, we establish and study the curvature properties of a $2$-parameter family of left invariant metrics $g_r$ on $Sp(2)$. Let $\mathfrak{sp}(2)$ be the Lie algebra of $Sp(2)$.
Given two real numbers $r_1,r_2>0$, we define a left invariant metric $g_r:=g_{(r_1,r_2)}$ on $Sp(2)$, such that on $$T_QSp(2)=\Big\{Q\Big(\begin{array}{cc}
          x& y\\
          -\overline{y}& z
        \end{array}\Big)~\Big|~\Big(\begin{array}{cc}
          x& y\\
          -\overline{y}& z
        \end{array}\Big)\in M(2,\mathbb{H}),\,\,\RE(x)=\RE(z)=0 \Big\},$$

\begin{equation}\label{left inv. metric on Sp2}
        \Big|Q\Big(\begin{array}{cc}
          x& y\\
          -\overline{y}& z
        \end{array}\Big)\Big|^2_{g_r}:=\frac{r_1}{2}|x|^2+|y|^2+\frac{r_2}{2}|z|^2.
\end{equation}
The left invariant metric $g_r$ is bi-invariant if and only if $r_1=r_2=1$.
As the first main result of this paper, we show

\begin{thm}\label{1.1}
For the left invariant metric $g_r$ on $Sp(2)$ defined in (\ref{left inv. metric on Sp2}).

(1).\label{sp2 nonnegative} The sectional curvature of the metric $g_r$ is nonnegative if and only if $r_1+r_2\leq 2$;

(2).\label{Einstein} The metric $g_r$ is Einstein if and only if $r_1=r_2=1$ or $\frac{1}{2}$.
\end{thm}

\begin{rem}
As a matter of fact, one can also discover the sufficient part of (1) from \cite{GZ00}, but here we succeed by a
direct calculation on the curvature.
\end{rem}



Taking the left invariant metric $g_r$ on $Sp(2)$ with $r_1+r_2\leq 2$, we define a product manifold $M^{14}=Sp(2)\times S^4$ with
the product metric, where $S^4\subset \mathbb{R}\oplus\mathbb{H}=\mathbb{R}^5$ is the unit sphere.
Consider the $S^{3}$-action on $M^{14}$ by
\begin{eqnarray}\label{S3 action on M14}
{\phi_1}: S^{3} \times M^{14} &\longrightarrow& M^{14} \nonumber\\
\Big(p, (Q, (t_1, t_2))\Big) &\longmapsto&\left(Q\Big(\begin{array}{cc}\overline{p}&\\ &1\end{array}\Big), (t_{1}, pt_2)\right),
\end{eqnarray}
where $p\in S^3$, $Q\in Sp(2)$, $(t_1, t_2)\in S^{4}\subset \mathbb{R}\oplus\mathbb{H}$. Since the action $\phi_1$ is an isometric, free action,
we acquire a quotient Riemannian manifold $\widetilde{N}^{11}:=M^{14}/\sim_{\phi_1}$, which is an $S^4$-bundle over $S^7$, and we have the following Riemannian submersion:
\begin{eqnarray}\label{pi1}
  \pi_1:\quad\, M^{14}\quad\, &\longrightarrow& \widetilde{N}^{11} \\
  (Q, (t_1, t_2)) &\longmapsto& [(Q, (t_1, t_2))] \nonumber
\end{eqnarray}
Moreover, the induced metric on $\widetilde{N}^{11}$ has nonnegative sectional curvature by the Gray-O'Neill formula.


Define an embedding by
\begin{eqnarray}\label{Sp2 into N11}
\Phi:  Sp(2)\times (0, \pi) &\longrightarrow& \quad\widetilde{N}^{11} \nonumber\\
~~~~  (Q, ~~\quad\theta) \quad &\longmapsto&  [(Q,(\cos\theta, \sin\theta))]
\end{eqnarray}
Noticing that there is naturally a cohomogeneity one action of $Sp(2)$ on $\widetilde{N}^{11}$ defined by
\begin{eqnarray}\label{Sp2 action on N11}
\Psi: \quad\, Sp(2)\,\,\times \,\,\widetilde{N}^{11}\,\quad & \longrightarrow& \widetilde{N}^{11}  \\
  (P, \,[(Q, (t_1, t_2))])& \longmapsto & [(PQ, (t_1, t_2))].\nonumber
\end{eqnarray}
The principal orbits are exactly $\Phi(Sp(2)\times \{\theta\})$ $(\theta\in (0,\pi))$. According to Proposition 2.8 in \cite{GT13}, they are isoparametric hypersurfaces in $\widetilde{N}^{11}$, and the singular orbits which are diffeomorphic to $S^7$ are the focal submanifolds. For convenience, we will denote $\Phi(Sp(2)\times \{\theta\})$ by $Sp(2)_{\theta}$ later on.

Moreover, there is an isometric involution on $\widetilde{N}^{11}$:
\begin{eqnarray*}
  \rho: \quad\widetilde{N}^{11}\quad & \longrightarrow & \widetilde{N}^{11}  \\
    \, [(Q, (\cos\theta, \sin\theta))]& \longmapsto & [\left(Q, (\cos (\pi-\theta), \sin (\pi-\theta))\right)].
\end{eqnarray*}
Thus the hypersurfaces $Sp(2)_{\theta}$ and $Sp(2)_{\pi-\theta}$ $(\theta\in (0,\pi))$ correspond to each other by $\rho$.

Furthermore, by analyzing the second fundamental form of the hypersurface $Sp(2)_{\theta}$ in $\widetilde{N}^{11}$, we obtain the second main result of the present paper:
\begin{thm}\label{Sp2 isoparametric in N11}
(1). The principal orbit $Sp(2)_{\theta}$ with $\theta\in(0,\pi)$ of the action (\ref{Sp2 action on N11}) is an isoparametric hypersurface in $\widetilde{N}^{11}$ with principal curvatures
$\frac{\cot\theta}{1+\frac{2}{r_1}\sin^2\theta}$
of multiplicity $3$ and $0$ of multiplicity $7$. In particular, the fixed point set of the involution $\rho$, i.e., the hypersurface $Sp(2)_{\frac{\pi}{2}}$ is totally geodesic in $\widetilde{N}^{11}$.

(2). The singular orbits of the action (\ref{Sp2 action on N11}), which are focal submanifolds of this isoparametric foliation, are diffeomorphic to $S^7$ and totally geodesic in $\widetilde{N}^{11}$. 
\end{thm}

As a matter of fact, we can define the explicit isoparametric function on $\widetilde{N}^{11}$ corresponding to the isoparametric foliation in the theorem above by
projecting the following function $f$ through the projection $\pi_1$ in (\ref{pi1}):
\begin{eqnarray}\label{f on M14}
  f: \quad\,M^{14}\quad\, &\longrightarrow& \mathbb{R} \\
  A=(Q, (t_1, t_2)) &\longmapsto& t_1.\nonumber
\end{eqnarray}
It is easy to see that $\nabla^M f~|_{A}=(0, |t_2|^2, -t_1t_2)$, and thus $|\nabla^M f|_M^2=|t_2|^2=1-f^2$. Namely, $f$ is a transnormal function on $M^{14}$.

Defining $\widetilde{F}: \widetilde{N}^{11}\rightarrow \mathbb{R}$ by $\widetilde{F}\circ \pi_1=f$. Observing that the vertical distribution of the Riemannian submersion $\pi_1$ in (\ref{pi1}) is
\begin{equation*}
  \mathcal{V}_{A}=\Big\{ \left(Q\Big(\begin{array}{cc}
                                  u &  \\
                                   & 0
                                \end{array}
  \Big), (0, -ut_2)\right)~\big|~u\in \mathbb{H}, \mathrm{Re}(u)=0\Big\},
\end{equation*}
we find that $\nabla^M f$ is orthogonal to $\mathcal{V}_{A}$, thus a horizontal vector field. $\nabla^{\widetilde{N}}\widetilde{F}$ is just the projection of $\nabla^M f$ on the base manifold $\widetilde{N}^{11}$ and further
$$|\nabla^{\widetilde{N}}\widetilde{F}|_{\widetilde{N}}^2=|\nabla^M f|_M^2=1-\widetilde{F}^2.$$
That is, $\widetilde{F}$ is transnormal on $\widetilde{N}^{11}$. Moreover,
\begin{prop}\label{isoparametric function F1}
The function $\widetilde{F}: \widetilde{N}^{11}\rightarrow \mathbb{R}$ defined by $\widetilde{F}\circ \pi_1=f$ is an isoparametric function on $\widetilde{N}^{11}$ which satisfies:
\begin{equation}\label{isop func F1}
  \Big\{\begin{array}{ll}
|\nabla^{\widetilde{N}} \widetilde{F}|_{\widetilde{N}}^2=1-\widetilde{F}^2 \\
\quad \Delta^{\widetilde{N}} \widetilde{F}\,=-\left(1+\frac{3}{1+\frac{2}{r_1}(1-\widetilde{F}^2)}\right)\widetilde{F}.
\end{array}
\end{equation}
The corresponding isoparametric foliation is exactly the one in Theorem \ref{Sp2 isoparametric in N11}.
\end{prop}

Now we consider the following $S^3$-action on $\widetilde{N}^{11}$:
\begin{eqnarray}\label{S3 action on N11}
  \phi_2: \,S^3 \times \widetilde{N}^{11}\,\, &\longrightarrow& \quad \widetilde{N}^{11} \\
  \Big(q, [(Q, (t_1, t_2))]\Big) &\longmapsto& [(\Big(\begin{array}{cc} q &\\ & q\end{array}\Big) Q, (t_1, t_2\bar{q}))] .\nonumber
\end{eqnarray}
Since the action $\phi_2$ is free and isometric, we obtain a quotient Riemannian manifold $N^8:=\widetilde{N}^{11}/\sim_{\phi_2}$ and a Riemannian submersion $\pi_2:~\widetilde{N}^{11}\rightarrow N^8$. It also follows from the Gray-O'Neill formula that the induced metric on $N^{8}$ has nonnegative sectional curvature. Restricting $\phi_2$ on the hypersurface $Sp(2)_{\theta}$ and denoting its quotient by $\Sigma^7_{\theta}$, one can find that it is actually the $S^3$ action on $Sp(2)$ as in (\ref{S3 action on sp2}). Thus $\Sigma^7_{\theta}$ is diffeomorphic
to the Gromoll-Meyer sphere $\Sigma^7$. To describe it more clearly, we draw a diagram as below:
\begin{eqnarray}\label{pi2}
  Sp(2)_{\theta} &\hookrightarrow& \widetilde{N}^{11} \\
 \pi_2|_{Sp(2)_{\theta}} \Big\downarrow && \Big\downarrow \pi_2 \nonumber\\
  \Sigma^7_{\theta} &\hookrightarrow& N^8 \nonumber
\end{eqnarray}

Define $F: N^8\rightarrow \mathbb{R}$ by $F\circ \pi_2=\widetilde{F}$. Noticing that $\nabla^{\widetilde{N}} \widetilde{F}$
is also a horizontal direction of the Riemannian submersion $\pi_2$, we establish

\begin{thm}\label{F transnormal}
(1). The function $F$ defined by $F\circ \pi_2=\widetilde{F}$ is a transnormal function on $N^8$ satisfying $|\nabla^NF|_N^2=1-F^2$. The transnormal hypersurfaces $\Sigma^7_{\theta}$ ($\theta\in (0, \pi)$) are diffeomorphic to the Gromoll-Meyer sphere $\Sigma^7$. Moreover, $\Sigma^7_{\frac{\pi}{2}}$ is totally geodesic in $N^8$. The focal submanifolds $F^{-1}(\pm 1)$ are diffeomorphic to $S^4$, and totally geodesic in $N^8$. 

(2).For each $\theta\in (0,\pi)$, the induced metric of $\Sigma^7_{\theta}$ in $N^8$ has positive Ricci curvature and quasi-positive sectional curvature simultaneously.
\end{thm}

\begin{rem}
The function $F$ is not isoparametric. More precisely, for $\theta\in (0,\pi)$ and $\theta\neq\frac{\pi}{2}$, the mean curvature function of the transnormal hypersurface $\Sigma^7_{\theta}$ is not constant with respect to the unit normal vector field $\frac{\nabla^NF}{|\nabla^NF|_N}$. For instance, the mean curvatures at $\pi_2([(\Big(\begin{array}{cc}
          1& 0\\
          0& 1
        \end{array}\Big), (\cos\theta, \sin\theta))])$ and $\pi_2([(\Big(\begin{array}{cc}
          i& 0\\
          0& 1
        \end{array}\Big), (\cos\theta, \sin\theta))])$
of $\Sigma^7_{\theta}$ are $3\mu$ and $3\mu-\frac{16\lambda\mu}{8\lambda+r_2}$ respectively, where
$\lambda=\frac{\sin^2\theta}{1+\frac{2}{r_1}\sin^2\theta}$ and $\mu=\frac{\cot\theta}{1+\frac{2}{r_1}\sin^2\theta}$.
\end{rem}
Noticing that $F$ is a Morse-Bott function on $N^8$ with critical set $S^4\sqcup S^4$, we remark that according to the fundamental construction in Theorem 1.1 of \cite{QT15}, there exists a metric on $N^8$ such that $F$ is an isoparametric function and the the focal submanifolds are still $S^4$ and totally geodesic. However, one cannot know more about the intrinsic geometric properties of $N^8$ and the isoparametric hypersurfaces.


This paper is organized as follows: in Section 2, we will discuss about the curvature property of the left-invariant metric $g_r$ on $Sp(2)$. Section 3 will be focused on the geometry of the isoparametric foliation in $\widetilde{N}^{11}$, and Section 4 will be concentrated on the geometry of the transnormal system in $N^8$.

\section{Left invariant metrics on $Sp(2)$} \label{Left invariant metrics on $Sp(2)$}

Firstly, we will study the connection of the metric $g_r$ defined in (\ref{left inv. metric on Sp2}).

Let $\xi_1,\xi_2$ be two left invariant vector fields on $Sp(2)$ such that at any point $Q\in Sp(2)$,
$\xi_i|_Q=Q\Big(\begin{array}{cc}
          x_i& y_i\\
          -\overline{y}_i& z_i
        \end{array}\Big)$, $i=1,2$.
Denote by $\nabla$ the Levi-Civita connection associated with $g_r$.
Concerning the connection $\nabla_{\xi_1}\xi_2$ is also a left invariant vector field,
we establish the following formula which generalizes Lemma 4.5 in \cite{GT13}:
\begin{lem}\label{connection in Sp2}
$$\nabla_{\xi_1}\xi_2=\frac{1}{2}[\xi_1,\xi_2]+D(\xi_1,\xi_2),$$
where $D(\xi_1,\xi_2)$ is a left invariant vector field with
$$D(\xi_1,\xi_2)|_Q=Q\left(\begin{array}{cc}
          0& D_{12}\\
          -\overline{D}_{12}& 0
        \end{array}\right)\Big|_Q,$$
        and $D_{12}(Q)=\frac{1}{2}(1-r_1)(x_1y_2+x_2y_1)+\frac{1}{2}(r_2-1)(y_1z_2+y_2z_1)$.\hfill $\Box$
\end{lem}

\begin{proof}
Let $\xi_3$ be a left invariant vector field on $Sp(2)$ such that at $Q\in Sp(2)$, $\xi_3|_Q=Q\Big(\begin{array}{cc}
          x_3& y_3\\
          -\overline{y}_3& z_3
        \end{array}\Big)$.
From the left invariance of the vector fields, it follows that the first three items on the right hand
side of Koszul formula below are vanishing:
\begin{eqnarray}\label{Koszul}
&&2\langle \nabla_{\xi_1}\xi_2, \xi_3\rangle\\
&=&\xi_1\langle\xi_2, \xi_3\rangle+\xi_2\langle\xi_3, \xi_1\rangle-\xi_3\langle\xi_1, \xi_2\rangle-
\langle\xi_1, [\xi_2, \xi_3]\rangle+\langle\xi_2, [\xi_3, \xi_1]\rangle+\langle\xi_3, [\xi_1, \xi_2]\rangle.\nonumber
\end{eqnarray}

Denote by $\xi_{ij}:=[\xi_i, \xi_j]$ the left invariant vector field such that at any point $Q$,
$\xi_{ij}|_Q=Q\Big(\begin{array}{cc}
          x_{ij}& y_{ij}\\
          -\overline{y}_{ij}& z_{ij}
        \end{array}\Big)$ for $i, j=1, 2, 3.$ A direct calculation leads to
\begin{eqnarray*}\label{xi{ij}}
x_{ij}&=&x_ix_j-x_jx_i+y_j\overline{y}_i-y_i\overline{y}_j  \\
y_{ij}&=&x_iy_j-x_jy_i+y_iz_j-y_jz_i \\
z_{ij}&=& z_iz_j-z_jz_i+\overline{y}_jy_i-\overline{y}_iy_j,\quad i,j=1,2,3.
\end{eqnarray*}

Since $\xi_1$, $\xi_2$, $\xi_3$ are left invariant, it follows that
\begin{equation*}
  \langle \xi_i, [\xi_j, \xi_k]\rangle = \frac{r_1}{2}\langle x_i, x_{jk}\rangle+\langle y_i, y_{jk}\rangle +\frac{r_2}{2}\langle z_i, z_{jk}\rangle.
\end{equation*}
Thus \begin{eqnarray*}
        && -\langle\xi_1, [\xi_2, \xi_3]\rangle+\langle\xi_2, [\xi_3, \xi_1]\rangle \\
       &=& \frac{r_1}{2}\big(\langle x_1, x_{32}\rangle+\langle x_2, x_{31}\rangle\big)+\big(\langle y_1, y_{32}\rangle+\langle y_2, y_{31}\rangle\big)+\frac{r_2}{2}\big(\langle z_1, z_{32}\rangle+\langle z_2, z_{31}\rangle\big) \\
       &=& \big\langle (1-r_1)(x_1y_2+x_2y_1)+(r_2-1)(y_1z_2+y_2z_1), y_3\big\rangle \\
       &=& 2\langle D_{12}, y_3\rangle,
     \end{eqnarray*}
which yields the connection formula directly.

By the way, we notice that $D(\xi_1, \xi_2)=D(\xi_2, \xi_1)$.
\end{proof}

Next, we compute the sectional curvature of $(Sp(2), g_r)$ and show the following results.
\vspace{5mm}

\noindent
\textbf{Theorem 1.1}
\emph{For the left invariant metric $g_r$ on $Sp(2)$ as above.}

\emph{
(1).The sectional curvature of the metric $g_r$ is nonnegative if and only if $r_1+r_2\leq 2$;}

\emph{
(2).The metric $g_r$ is Einstein if and only if $r_1=r_2=1$ or $\frac{1}{2}$.}

\vspace{5mm}
\noindent
\textbf{Proof of (1):}

Given two left invariant vector fields $\xi_1$, $\xi_2$, by Lemma \ref{connection in Sp2},
\begin{eqnarray*}
 && \langle R(\xi_1, \xi_2)\xi_1, \xi_2\rangle \\
 &=& \big\langle \nabla_{\xi_2}\nabla_{\xi_1}\xi_1-\nabla_{\xi_1}\nabla_{\xi_2}\xi_1+\nabla_{[\xi_1, \xi_2]}\xi_1, \xi_2\big\rangle \\
 &=& \xi_2\langle\nabla_{\xi_1}\xi_1, \xi_2\rangle-\xi_1\langle\nabla_{\xi_2}\xi_1, \xi_2\rangle-\langle \nabla_{\xi_1}\xi_1, \nabla_{\xi_2}\xi_2\rangle
 + \langle \nabla_{\xi_2}\xi_1, \nabla_{\xi_1}\xi_2\rangle+\langle \nabla_{[\xi_1, \xi_2]}\xi_1, \xi_2\rangle\\
 &=& \xi_2\langle D(\xi_1, \xi_1), \xi_2\rangle-\xi_1\langle \frac{1}{2}[\xi_2, \xi_1]+D(\xi_2, \xi_1), \xi_2\rangle-\langle D(\xi_1, \xi_1), D(\xi_2, \xi_2)\rangle\\
 &&+\big\langle \frac{1}{2}[\xi_2, \xi_1]+D(\xi_2, \xi_1), \frac{1}{2}[\xi_1, \xi_2]+D(\xi_1, \xi_2)\big\rangle+\big\langle \frac{1}{2}\big[[\xi_1, \xi_2], \xi_1\big]+D([\xi_1,\xi_2], \xi_1), \xi_2\big\rangle
\end{eqnarray*}
 It is easily seen that $\xi_2\langle D(\xi_1, \xi_1), \xi_2\rangle=0$ and $\xi_1\langle \frac{1}{2}[\xi_2, \xi_1]+D(\xi_2, \xi_1), \xi_2\rangle=0$.
Denote
\begin{eqnarray*}
  R_1 &:=& -\langle D(\xi_1, \xi_1), D(\xi_2, \xi_2)\rangle \\
  R_2 &:=& \big\langle \frac{1}{2}[\xi_2, \xi_1]+D(\xi_2, \xi_1), \frac{1}{2}[\xi_1, \xi_2]+D(\xi_1, \xi_2)\big\rangle\\
  &=&-\frac{1}{4}|[\xi_1, \xi_2]|^2+|D(\xi_1, \xi_2)|^2 \\
  R_3 &:=& \big\langle \frac{1}{2}\big[[\xi_1, \xi_2], \xi_1\big]+D([\xi_1,\xi_2], \xi_1), \xi_2\big\rangle,
\end{eqnarray*}
then $\langle R(\xi_1, \xi_2)\xi_1, \xi_2\rangle=R_1+R_2+R_3$.

For convenience, we denote
\begin{equation*}
\begin{array}{ccc}
      \alpha_1:=y_2\overline{y}_1-y_1\overline{y}_2 & \beta_1:=x_1x_2-x_2x_1 & \gamma_1:=x_1y_2-x_2y_1 \\
      \alpha_2:=\overline{y}_2y_1-\overline{y}_1y_2 & \beta_2:=z_1z_2-z_2z_1 & \gamma_2:=y_1z_2-y_2z_1,
\end{array}
\end{equation*}
thus
\begin{equation*}
  x_{12}=\alpha_1+\beta_1, \quad y_{12}=\gamma_1+\gamma_2, \quad z_{12}=\alpha_2+\beta_2.
\end{equation*}
Using the notations in Lemma \ref{connection in Sp2}, we derive that
\begin{eqnarray*}
  R_1 &=& -\big\langle (1-r_1)x_1y_1+(r_2-1)y_1z_1, (1-r_1)x_2y_2+(r_2-1)y_2z_2\big\rangle\\
  R_2 &=& -\frac{1}{4}|[\xi_1, \xi_2]|^2+|D(\xi_1, \xi_2)|^2, \\
  &=&-\frac{1}{4}|[\xi_1, \xi_2]|^2+|\frac{1}{2}(1-r_1)(x_1y_2+x_2y_1)+\frac{1}{2}(r_2-1)(y_1z_2+y_2z_1)|^2\\
  R_3 &=& \frac{1}{2}\big\langle \big[[\xi_1, \xi_2], \xi_1\big], \xi_2\big\rangle+\langle D([\xi_1,\xi_2], \xi_1), \xi_2\big\rangle\\
  &=& \frac{1}{2}|[\xi_1, \xi_2]|^2+\frac{1-r_1}{4}\big(\langle x_{12}, y_2\overline{y}_1-y_1\overline{y}_2\rangle+\langle y_{12}, 2x_2y_1\rangle\big)\\
   &&+\frac{r_2-1}{4}\big(\langle y_{12}, 2y_1z_2\rangle+\langle z_{12}, \overline{y}_1y_2-\overline{y}_2y_1\rangle\big)\\
  &&+\frac{1-r_1}{2}\big(\langle x_{12}, y_2\overline{y}_1\rangle-\langle y_{12}, x_1y_2\rangle\big)
  +\frac{r_2-1}{2}\big(\langle z_{12}, \overline{y}_1y_2\rangle-\langle y_{12}, y_2z_1\rangle\big)\\
  &=& \frac{1}{2}|[\xi_1, \xi_2]|^2+\frac{1-r_1}{2}\big(\langle \alpha_1+\beta_1, \alpha_1\rangle+\langle \gamma_1+\gamma_2, -\gamma_1\rangle\big)\\
  &&+  \frac{r_2-1}{2}\big(\langle \gamma_1+\gamma_2, \gamma_2\rangle+\langle \alpha_2+\beta_2, -\alpha_2\rangle\big)
\end{eqnarray*}

Therefore,
\begin{eqnarray}\label{sectional cuervature of xi1 xi2}
  &&\langle R(\xi_1, \xi_2)\xi_1, \xi_2\rangle \nonumber \\
 &=&  \frac{1}{4}\big|[\xi_1, \xi_2]\big|^2+|D(\xi_1, \xi_2)|^2-\langle D(\xi_1, \xi_1), D(\xi_2, \xi_2)\rangle \nonumber\\
 &&+\langle \alpha_1+\beta_1, \frac{1-r_1}{2}\alpha_1\rangle+\langle \gamma_1+\gamma_2, \frac{r_1-1}{2}\gamma_1\rangle \nonumber\\
 && +\langle \gamma_1+\gamma_2, \frac{r_2-1}{2}\gamma_2\rangle+\langle \alpha_2+\beta_2, \frac{1-r_2}{2}\alpha_2\rangle \nonumber\\
 &=&  \frac{r_1}{8}|\alpha_1+\beta_1|^2+\frac{1}{4}|\gamma_1+\gamma_2|^2+\frac{r_2}{8}|\alpha_2+\beta_2|^2 \nonumber\\
 &&+|\frac{1}{2}(1-r_1)(x_1y_2+x_2y_1)+\frac{1}{2}(r_2-1)(y_1z_2+y_2z_1)|^2 \nonumber\\
 &&-\big\langle (1-r_1)x_1y_1+(r_2-1)y_1z_1, (1-r_1)x_2y_2+(r_2-1)y_2z_2\big\rangle \nonumber\\
 &&+\langle \alpha_1+\beta_1, \frac{1-r_1}{2}\alpha_1\rangle+\langle \gamma_1+\gamma_2, \frac{r_1-1}{2}\gamma_1\rangle \nonumber\\
 && +\langle \gamma_1+\gamma_2, \frac{r_2-1}{2}\gamma_2\rangle+\langle \alpha_2+\beta_2, \frac{1-r_2}{2}\alpha_2\rangle \nonumber\\
  &=& \frac{1}{4}|r_1\gamma_1+r_2\gamma_2|^2+\frac{r_1}{8}|\beta_1+(3-2r_1)\alpha_1|^2+\frac{r_2}{8}|\beta_2+(3-2r_2)\alpha_2|^2\\
  && + \frac{1}{2}\big((1-r_1)^3+(1-r_2)^3\big)|\alpha_1|^2, \nonumber
\end{eqnarray}
where in the last equality we used $|\alpha_1|=|\alpha_2|$.

Next we prove the sufficiency and necessity for $(Sp(2),g_r)$ to be non-negatively curved. Suppose $r_1+r_2\leq2$. It follows that
$(1-r_1)^3+(1-r_2)^3\geq 0$ and every term in the curvature formula is non-negative. This proves the sufficiency.
As for the necessity, we claim that there always exist two vectors $\xi_p$ $(p=1,2)$ such that $|\alpha_1|^2\neq 0$ and $\langle R(\xi_1,\xi_2)\xi_1,\xi_2\rangle=\frac{1}{2}\Big((1-r_1)^3+(1-r_2)^3\Big)|\alpha_1|^2<0$ if $r_1+r_2>2$.
In fact, by a long but straightforward calculation, we can choose
$$\xi_1=\Big(\begin{array}{cc}
tu\oi &-\oj\\
-\oj&tv\oi
\end{array}\Big), \quad \xi_2=\Big(\begin{array}{cc}
t(r_2-u)\oj &\oi\\
\oi&t(r_1-v)\oj
\end{array}\Big),$$
where $t$ is any sufficiently large number and $u,v$ are any real solutions for the quadratic equations
$t^2u(r_2-u)=2r_1-3$ and $t^2v(r_1-v)=2r_2-3$, respectively.
Now one can verify directly that for these two vectors $\xi_1,\xi_2$,
$$\alpha_1\neq0, \quad r_1\gamma_1+r_2\gamma_2=0, \quad \beta_p+(3-2r_p)\alpha_p=0, \quad p=1,2.$$
The proof is now complete.

\hfill $\Box$


We are now in a position to prove the second part of Theorem \ref{Einstein}.

\noindent
\textbf{Proof of (2):}

Under the metric (\ref{left inv. metric on Sp2}), we can fix an orthonormal frame $\{e_p\in\mathfrak{sp}(2)\mid p=1,\ldots,10\}$ as follows:
\begin{equation*}\label{basis-sp2}
\begin{array}{llll}
e_1=\sqrt{\frac{2}{r_1}} \Big(\begin{array}{cc}
\oi & 0\\
0&0
\end{array}\Big), & e_2=\sqrt{\frac{2}{r_1}} \Big(\begin{array}{cc}
\oj & 0\\
0&0
\end{array}\Big), & e_3=\sqrt{\frac{2}{r_1}} \Big(\begin{array}{cc}
\ok & 0\\
0&0
\end{array}\Big),&\\
e_4= \Big(\begin{array}{cc}
0 & 1\\
-1&0
\end{array}\Big),&e_5=\Big(\begin{array}{cc}
0 & \oi\\
\oi&0
\end{array}\Big),&e_6=\Big(\begin{array}{cc}
0 & \oj\\
\oj&0
\end{array}\Big),&e_7=\Big(\begin{array}{cc}
0 & \ok\\
\ok&0
\end{array}\Big),\\
e_8=\sqrt{\frac{2}{r_2}} \Big(\begin{array}{cc}
0 & 0\\
0& \oi
\end{array}\Big), & e_9=\sqrt{\frac{2}{r_2}} \Big(\begin{array}{cc}
0 & 0\\
0& \oj
\end{array}\Big), & e_{10}=\sqrt{\frac{2}{r_2}} \Big(\begin{array}{cc}
0 & 0\\
0& \ok
\end{array}\Big).&
\end{array}
\end{equation*}

We first consider the necessary condition for $(Sp(2), g_r)$ to be Einstein.

\noindent
$(1)$ For $e_1=\sqrt{\frac{2}{r_1}} \Big(\begin{array}{cc}
\oi & 0\\
0&0
\end{array}\Big)$, we can immediately find that $\alpha_1=\alpha_2=\beta_2=\gamma_2=0$.
Moreover,
\begin{eqnarray*}
  \beta_1&=&\sqrt{\frac{2}{r_1}}(\oi x_p-x_p\oi)=\left\{\begin{array}{cc}
  \frac{4}{r_1}\ok& p=2\\
  -\frac{4}{r_1}\oj& p=3\\
  0& p=4,\dots,10
  \end{array}\right.\\
  \gamma_1&=&\sqrt{\frac{2}{r_1}}y_p\oi=\left\{\begin{array}{cc}
  0& p=2,3,8,9,10\\
  \sqrt{\frac{2}{r_1}}\oi& p=4\\
  -\sqrt{\frac{2}{r_1}}& p=5\\
  \sqrt{\frac{2}{r_1}}\ok&p=6\\
  -\sqrt{\frac{2}{r_1}}\oj&p=7
\end{array}\right.
\end{eqnarray*}
Therefore, from (\ref{sectional cuervature of xi1 xi2}), it follows that the sectional curvature of the plane spanned by $e_1$ and $e_p$ is
\begin{equation}\label{Ke1 ep}
K(e_1, e_p)=\frac{1}{4}r_1^2|\gamma_1|^2+\frac{r_1}{8}|\beta_1|^2=\left\{\begin{array}{ccl}
\frac{2}{r_1} & & p=2, 3;\\
\frac{r_1}{2} & & p=4,5,6,7;\\
0 && p=8,9,10.
\end{array}\right.
\end{equation}
Hence $Ric(e_1)=2r_1+\frac{4}{r_1}$. Similarly, $Ric(e_2)=Ric(e_3)=2r_1+\frac{4}{r_1}$.

\noindent
$(2)$ For $e_4=\Big(\begin{array}{cc}
0 & 1\\
-1&0
\end{array}\Big)$, it is easily observe that $\beta_1=\beta_2=0$ for any $p$.
Moreover,
\begin{eqnarray*}
\alpha_1&=& y_p\overline{y}_4-y_4\overline{y}_p=2\mathrm{Im} y_p=\left\{\begin{array}{lcl} 0&& p=1,2,3,8,9,10.\\ 2\oi&& p=5\\ 2\oj&& p=6\\ 2\ok& & p=7.\end{array}\right. \\
\alpha_2&=& \overline{y}_p y_4-\overline{y}_4y_p=-2\mathrm{Im} y_p=-\alpha_1 \\
\gamma_1&=&x_4y_p-x_py_4=-x_p=\left\{\begin{array}{lcl} -\sqrt{\frac{2}{r_1}} \oi&& p=1\\ -\sqrt{\frac{2}{r_1}} \oj&& p=2\\ -\sqrt{\frac{2}{r_1}} \ok&& p=3\\ 0& &p=5,\cdots, 10.\end{array}\right. \\
\gamma_2&=& y_4z_p-y_pz_4=z_p=\left\{\begin{array}{rcl} 0&& p=1,2,3,5,6,7.\\ \sqrt{\frac{2}{r_2}}\oi&& p=8\\ \sqrt{\frac{2}{r_2}}\oj&& p=9\\ \sqrt{\frac{2}{r_2}}\ok&& p=10.\end{array}\right.
\end{eqnarray*}
Therefore, from (\ref{sectional cuervature of xi1 xi2}), it follows that the sectional curvature of the plane spanned by $e_4$ and $e_p$ is
\begin{eqnarray}\label{Ke4 ep}
  K(e_4,e_p) &=& \frac{1}{4}|r_1\gamma_1+r_2\gamma_2|^2+\frac{r_1}{8}|(3-2r_1)\alpha_1|^2+\frac{r_2}{8}|(3-2r_2)\alpha_2|^2 \\
   && +\frac{1}{2}\big((1-r_1)^3+(1-r_2)^3\big)|\alpha_1|^2\nonumber\\
   &=& \left\{\begin{array}{ccl}
\frac{r_1}{2} & & p=1, 2, 3;\\
4-\frac{3}{2}(r_1+r_2) & & p=5,6,7;\\
\frac{r_2}{2} & & p=8,9,10.
\end{array}\right.\nonumber
\end{eqnarray}
Hence $Ric(e_4)=12-3(r_1+r_2)$. Similarly, $Ric(e_5)=Ric(e_6)=Ric(e_7)=12-3(r_1+r_2)$.

$(3)$ For $e_8=\sqrt{\frac{2}{r_2}} \Big(\begin{array}{cc}
0 & 0\\
0& \oi
\end{array}\Big)$, it is easily observed that $\alpha_1=\alpha_2=\beta_1=\gamma_1=0$ for any $p$.
Moreover,
\begin{eqnarray*}
\beta_2&=& z_8z_p-z_pz_8=\sqrt{\frac{2}{r_2}}(\oi z_p-z_p\oi)=\left\{\begin{array}{ccl} 0& &p=1,\cdots,7.\\ 2\sqrt{\frac{2}{r_2}}\ok& & p=9\\ -2\sqrt{\frac{2}{r_2}}\oj& & p=10.\end{array}\right. \\
\gamma_2&=& y_8z_p-y_pz_8=-y_pz_8=\left\{\begin{array}{ccl} 0&& p=1,2,3,8,9,10;\\ -\sqrt{\frac{2}{r_2}}\oi&& p=4;\\ \sqrt{\frac{2}{r_2}}& &p=5;\\ \sqrt{\frac{2}{r_2}}\ok& &p=6;\\-\sqrt{\frac{2}{r_2}}\oj&& p=7.\end{array}\right.
\end{eqnarray*}
Therefore, from (\ref{sectional cuervature of xi1 xi2}), it follows that the sectional curvature of the plane spanned by $e_8$ and $e_p$ is
\begin{equation}\label{Ke8 ep}
K(e_8,e_p)=\frac{1}{4}r_2^2|\gamma_2|^2+\frac{r_2}{8}|\beta_2|^2=\left\{\begin{array}{ccl}
0 & & p=1, 2, 3;\\
\frac{r_2}{2} & & p=4,5,6,7;\\
\frac{2}{r_2}& & p=9,10.
\end{array}\right.
\end{equation}
Hence $Ric(e_8)=2r_2+\frac{4}{r_2}$. Similarly, $Ric(e_9)=Ric(e_{10})=2r_2+\frac{4}{r_2}$.

In summary of $(1), (2), (3)$,
\begin{equation}\label{Ric curvature}
  Ric(e_p)=\left\{\begin{array}{ccl}
2r_1+\frac{4}{r_1} & &p=1,2,3;\\
12-3(r_1+r_2) && p=4,5,6,7;\\
2r_2+\frac{4}{r_2} && p=8,9,10.
\end{array}\right.
\end{equation}

Therefore, $(Sp(2), g_r)$ is Einstein implies that $2r_1+\frac{4}{r_1}=2r_2+\frac{4}{r_2}=12-3(r_1+r_2)$, which further implies $r_1=r_2=1$ or $\frac{1}{2}$.


Conversely, notice that when $r_1=r_2=1$, the metric $g_r$ is bi-invariant, and $(Sp(2), g_r)$ is Einstein.
So we next prove that $r_1=r_2=\frac{1}{2}$ is a sufficient condition for $(Sp(2), g_r)$ to be Einstein. We will need the following
observation to deal with the case.
\begin{eqnarray*}
(Sp(2), g_r)~ \mathrm{is~ Einstein} &\Longleftrightarrow& \exists ~\mathrm{constant}~ c, \mathrm{~such~ that~} Ric(e_i)=c ~~~\forall~ 1\leq i\leq 10,\\
&& \mathrm{and}~\forall~ i\neq j, \sum_{k\neq i,j}K(\frac{e_i+e_j}{\sqrt{2}}, e_k)=c-K(e_i, e_j).
\end{eqnarray*}

In our case $r_1=r_2=\frac{1}{2}$, from (\ref{Ric curvature}), we derive $c=9$. By the results in (\ref{Ke1 ep}), (\ref{Ke4 ep}), (\ref{Ke8 ep}), we need only to calculate $K(\frac{e_i+e_j}{\sqrt{2}}, e_k)=\frac{1}{2}\langle R(e_i+e_j, e_k)(e_i+e_j), e_k\rangle$ with respect to the orthonormal basis:
\begin{equation*}
\begin{array}{cccc}
e_1=2 \Big(\begin{array}{cc}
\oi & 0\\
0&0
\end{array}\Big), & e_2=2 \Big(\begin{array}{cc}
\oj & 0\\
0&0
\end{array}\Big), & e_3=2 \Big(\begin{array}{cc}
\ok & 0\\
0&0
\end{array}\Big),&\\
e_4= \Big(\begin{array}{cc}
0 & 1\\
-1&0
\end{array}\Big),&e_5=\Big(\begin{array}{cc}
0 & \oi\\
\oi&0
\end{array}\Big),&e_6=\Big(\begin{array}{cc}
0 & \oj\\
\oj&0
\end{array}\Big),&e_7=\Big(\begin{array}{cc}
0 & \ok\\
\ok&0
\end{array}\Big),\\
e_8=2 \Big(\begin{array}{cc}
0 & 0\\
0& \oi
\end{array}\Big), & e_9=2\Big(\begin{array}{cc}
0 & 0\\
0& \oj
\end{array}\Big), & e_{10}=2 \Big(\begin{array}{cc}
0 & 0\\
0& \ok
\end{array}\Big).&
\end{array}
\end{equation*}
In fact, it is a direct calculation using the definitions of $\alpha_1, \alpha_2, \beta_1, \beta_2, \gamma_1, \gamma_2$ and the expression of curvature (\ref{sectional cuervature of xi1 xi2}). We will omit the details here.

\hfill $\Box$

\section{geometry of the isoparametric foliation on $\widetilde{N}^{11}$}



Recall the $S^{3}$ action on $M^{14}$ by (\ref{S3 action on M14})
\begin{eqnarray*}
{\phi_1}: S^{3} \times M^{14} &\longrightarrow& M^{14} \nonumber\\
\Big(p,(Q,(t_1, t_2))\Big) &\longmapsto&\left(Q\left(\begin{array}{cc}\overline{p}&\\ &1\end{array}\right), (t_{1}, pt_2)\right),
\end{eqnarray*}
where $p\in S^3$, $Q\in Sp(2)$, $(t_1, t_2)\in S^{4}\subset \mathbb{R} \oplus \mathbb{H}$.

Observe that $\phi_1$ is an isometric free action.
We will study the induced metric on the quotient space $\widetilde{N}^{11}$. Consider the embedding (\ref{Sp2 into N11}):
\begin{eqnarray*}
\Phi:  Sp(2)\times (0, \pi) &\longrightarrow& \quad\widetilde{N}^{11} \nonumber\\
~~~~  (Q, ~~\quad\theta) \quad &\longmapsto&  [(Q,(\cos\theta, \sin\theta))]
\end{eqnarray*}
At any point $A=(Q,\theta)$, and $X=(Q\xi, c)\in T_A(Sp(2)\times (0,\pi))$, we will consider the metric:
\begin{equation*}
|X|^2:=|\Phi_{\ast}X|_{\widetilde{N}}^2.
\end{equation*}
So our next task is to calculate $|\Phi_{\ast}X|_{\widetilde{N}}^2$ by virtue of the metric induced from the Riemannian submersion $\pi_1$.

At the point $B=(Q, (\cos\theta, \sin\theta))\in M^{14}$, $\widetilde{\Phi_{\ast}X}=(Q\xi, (-c\sin\theta, c\cos\theta))$ is a lift of $\Phi_{\ast}X$. To compute $|\widetilde{\Phi_{\ast}X}^{\mathcal{H}}|_{M}^2$, we need to
decompose $\widetilde{\Phi_{\ast}X}$ along the vertical distribution $\mathcal{V}$ and the horizontal distribution $\mathcal{H}$. In fact, when $\cos\theta\neq 0$, the vertical distribution is
\begin{equation*}
  \mathcal{V}_B=\Big\{\left( -Q\left(\begin{array}{cc}
                                  u &  \\
                                   & 0
                                \end{array}
  \right), (0, \sin\theta~u)\right)~\big|~u\in \mathbb{H}, \mathrm{Re}(u)=0\Big\}
\end{equation*} and the horizontal distribution is
\begin{equation*}
  \mathcal{H}_B=\Big\{\left( Q\left(\begin{array}{cc}
                                  \frac{2}{r_1}\sin\theta~\mathrm{Im}(v) & y \\
                                  -\overline{y} & z
                                \end{array}
  \right), (-\frac{\sin\theta}{\cos\theta}~\mathrm{Re}(v), v)\right)~\big|~v, y, z\in \mathbb{H}, \mathrm{Re}(z)=0\Big\}.
\end{equation*}
When $\cos\theta= 0$, the vertical and horizontal distributions are
\begin{eqnarray*}
  \mathcal{V}_B&=&\Big\{\left( -Q\left(\begin{array}{cc}
                                  u &  \\
                                   & 0
                                \end{array}
  \right), (0, u)\right)~\big|~u\in \mathbb{H}, \mathrm{Re}(u)=0\Big\}\\
  \mathcal{H}_B&=&\Big\{\left( Q\left(\begin{array}{cc}
                                  \frac{2}{r_1} v & y \\
                                  -\overline{y} & z
                                \end{array}
  \right), (t, v)\right)~\big|~t\in \mathbb{R}, v, y, z\in \mathbb{H}, \mathrm{Re}(v)=\mathrm{Re}(z)=0\Big\}.
\end{eqnarray*}
Decomposing $\widetilde{\Phi_{\ast}X}=\widetilde{\Phi_{\ast}X}^{\mathcal{V}}+\widetilde{\Phi_{\ast}X}^{\mathcal{H}}$, we see that
\begin{equation*}
  (\widetilde{\Phi_{\ast}X})^{\mathcal{H}}=\left( Q\left(\begin{array}{cc}
                                  \frac{2}{r_1}\sin^2\theta~u & y \\
                                  -\overline{y} & z
                                \end{array}
  \right), (-c\sin\theta, c\cos\theta+\sin\theta u)\right) ~\mathrm{with}~u=\frac{x}{1+\frac{2}{r_1}\sin^2\theta},
\end{equation*}
and thus
\begin{eqnarray*}\label{induced metric on N11}
  |X|^2&=&|\widetilde{\Phi_{\ast}X}^{\mathcal{H}}|^2=\frac{\sin^2\theta}{1+\frac{2}{r_1}\sin^2\theta}|x|^2+|y|^2+\frac{r_2}{2}|z|^2+c^2\\
  &:=& \lambda(\theta)|x|^2+|y|^2+\frac{r_2}{2}|z|^2+c^2.
\end{eqnarray*}
Besides, by the Gray-O'Neill formula, the sectional curvature of $\widetilde{N}^{11}$ is non-negative.

Now we are in position to investigate the orbit geometry on $\widetilde{N}^{11}$.
\vspace{5mm}

\noindent
\textbf{Theorem 1.3.}
\emph{(1). The principal orbit $Sp(2)_{\theta}$ with $\theta\in(0,\pi)$ of the action (\ref{Sp2 action on N11}) is an isoparametric hypersurface in $\widetilde{N}^{11}$ with principal curvatures
$\frac{\cot\theta}{1+\frac{2}{r_1}\sin^2\theta}$
of multiplicity $3$ and $0$ of multiplicity $7$. In particular, the fixed point of the involution $\rho$, i.e., the hypersurface $Sp(2)_{\frac{\pi}{2}}$ is totally geodesic in $\widetilde{N}^{11}$.}

\emph{(2). The singular orbits of the action (\ref{Sp2 action on N11}), which are focal submanifolds of this isoparametric foliation, are diffeomorphic to $S^7$ and totally geodesic in $\widetilde{N}^{11}$. }

\vspace{5mm}
\noindent
\textbf{Proof: }

As we mentioned in the introduction, we can construct an isoparametric foliation on $\widetilde{N}^{11}$ from the point view of cohomogeneity one action, that is, the action of $Sp(2)$ on $\widetilde{N}^{11}$ (\ref{Sp2 action on N11})
\begin{eqnarray*}
\Psi: \quad\, Sp(2)\,\,\times \,\,\widetilde{N}^{11}\,\quad & \longrightarrow& \widetilde{N}^{11}  \\
  P, \,[(Q, (\cos\theta, \sin\theta))]& \longmapsto & [(PQ, (\cos\theta, \sin\theta))] .\nonumber
\end{eqnarray*}
The principal orbits $Sp(2)_{\theta}=\Phi(Sp(2)\times \{\theta\})$, which are diffeomorphic to $Sp(2)$, are isoparametric hypersurfaces in $\widetilde{N}^{11}$. So we only need to calculate the principal curvatures of $Sp(2)_{\theta}$.

For any vector fields $X_1=(Q\xi_1, c_1)$, $X_2=(Q\xi_2, c_2)$ with $\xi_i~|_Q=Q\left(\begin{array}{cc}
                                                                               x_i & y_i \\
                                                                               -\overline{y}_i & z_i
                                                                             \end{array}
\right)$ and $c_i\in \mathbb{R}$ ($i=1, 2$) on $\widetilde{N}^{11}$, we calculate the connection $\nabla^{\widetilde{N}}_{X_1}X_2$ in a similar way
as in Section 2 and get
\begin{equation}\label{connection on N11}
  \nabla^{\widetilde{N}}_{X_1}X_2=\frac{1}{2}\big[X_1, X_2\big]+E(X_1, X_2),
\end{equation}
where
\begin{eqnarray*}
E(X_1, X_2)&=&\left( Q\left(\begin{array}{cc}
                     E_{11} & E_{12} \\
                     -\overline{E}_{12} & 0
                   \end{array}
 \right), -\frac{\lambda'}{2}\langle x_1, x_2\rangle\right) ~~\mathrm{with}\\
 E_{11}&=&\frac{\lambda'}{2\lambda}(c_1x_2+c_2x_1)\\
 E_{12}&=& (\frac{1}{2}-\lambda)(x_1y_2+x_2y_1)+\frac{r_2-1}{2}(y_1z_2+y_2z_1).
\end{eqnarray*}
Taking the normal direction of the hypersurface $Sp(2)_{\theta}$ as $N=(0, -1)$, by virtue
of $\big[X_1, X_2\big]=(Q\big[\xi_1, \xi_2\big], 0)$, we obtain the second fundamental form of $Sp(2)_{\theta}$:
$B(X_1, X_2)=\frac{\lambda'}{2}\langle x_1, x_2\rangle$.
Thus $X=\Big( Q\Big(\begin{array}{cc}
   x & 0 \\
   0 & 0
\end{array}
\Big),\, 0\Big)$ with $x\in \mathbb{H}$ and $\mathrm{Re}(x)=0$ is an eigenvector of the shape operator and the corresponding
principal curvature is $\frac{\lambda'}{2\lambda}=\frac{\cot\theta}{1+\frac{2}{r_1}\sin^2\theta}$ with multiplicity $3$. The other
eigenvalue is $0$ with multiplicity $7$. In particular, when $\theta=\frac{\pi}{2}$, $\lambda'=0$, which means that the hypersurface $Sp(2)_{\frac{\pi}{2}}$ is totally geodesic.
Moreover, the mean curvature of $Sp(2)_{\theta}$ in $\widetilde{N}^{11}$ is
$H=\frac{3\lambda'}{2\lambda}.$

Clearly, the singular orbits of the action (\ref{Sp2 action on N11}) are $\{[(Q, \pm 1, 0)]~|~Q\in Sp(2)\}$, i.e., the projection of $Sp(2)\times \{(\pm 1, 0)\}$ under the Riemannian submersion $\pi_1$. For convenience, we only focus on one of the focal submanifolds $\{[(Q, (1, 0))]~|~Q\in Sp(2)\}$. From the action (\ref{S3 action on M14}), we know that $(Q, (1, 0))\sim (Q\Big(\begin{array}{cc}
                                                                                         \overline{p} &  \\
                                                                                          & 1
                                                                                       \end{array}
\Big), (1, 0))$, which induces an isometric free action
\begin{eqnarray*}\label{S3 action on Sp2}
{\phi'_1}: S^{3} \times Sp(2) &\longrightarrow& Sp(2) \nonumber\\
p,\qquad Q\, &\longmapsto& Q\left(\begin{array}{cc}\overline{p}&\\ &1\end{array}\right),
\end{eqnarray*}
and furthermore a Riemannian submersion:
\begin{eqnarray*}\label{S3 action on Sp2}
\pi'_1: \quad Sp(2)\quad &\longrightarrow& \quad S^7 \nonumber\\
Q=\left(\begin{array}{cc}a&b\\ c&d\end{array}\right) &\longmapsto& \left(\begin{array}{c}b\\ d\end{array}\right).
\end{eqnarray*}
Through $Q$, we have the vertical and horizontal distributions as follows
\begin{eqnarray*}
  \mathcal{V}_Q&=&\Big\{ Q\left(\begin{array}{cc}
                                  u &  \\
                                   & 0
                                \end{array}
  \right)~\big|~u\in \mathbb{H}, \mathrm{Re}(u)=0\Big\}\\
  \mathcal{H}_Q&=&\Big\{Q\left(\begin{array}{cc}
                                  0 & y \\
                                  -\overline{y} & z
                                \end{array}
  \right)~\big|~y, z\in \mathbb{H}, \mathrm{Re}(z)=0\Big\}.
\end{eqnarray*}
Thus for any $Q\xi\in T_QSp(2)$, if $Q\xi\in \mathcal{H}_Q$, then by the left invariant metric (\ref{left inv. metric on Sp2}) on $Sp(2)$,
we obtain
\begin{equation*}
  \big|Q\xi\big|^2=|y|^2+\frac{r_2}{2}|z|^2.
\end{equation*}
Since $r_2$ is less than $2$ in our case, the focal submanifold $S^7$ is not a round sphere.

Observe that $S^7$ is a submanifold of $\widetilde{N}^{11}$. Moreover, since the inverse images
$\pi_1^{-1}S^7\cong Sp(2)$ of the focal submanifolds are totally geodesic in $Sp(2)\times S^4$, it follows that the focal submanifolds are totally geodesic in $\widetilde{N}^{11}$ (See Proposition 7.1 in \cite{TT95}).

\hfill $\Box$

As a matter of fact, we can give the explicit isoparametric function corresponding to this isoparametric foliation on $\widetilde{N}^{11}$.

\noindent
\textbf{Proposition 1.4.}
\emph{The function $\widetilde{F}: \widetilde{N}^{11}\rightarrow \mathbb{R}$ defined by $\widetilde{F}\circ \pi_1=f$ is an isoparametric function on $\widetilde{N}^{11}$ which satisfies:
\begin{equation}\label{isop func F1}
  \Big\{\begin{array}{ll}
|\nabla^{\widetilde{N}} \widetilde{F}|_{\widetilde{N}}^2=1-\widetilde{F}^2 \\
\quad \Delta^{\widetilde{N}} \widetilde{F}\,=-\left(1+\frac{3}{1+\frac{2}{r_1}(1-\widetilde{F}^2)}\right)\widetilde{F}.
\end{array}
\end{equation}
The corresponding isoparametric foliation is exactly the one in Theorem \ref{Sp2 isoparametric in N11}.}

\vspace{5mm}
\noindent
\textbf{Proof:}

Recall $f$ that defined in (\ref{f on M14}). Project $f$ onto $\widetilde{N}^{11}$ and define $\widetilde{F}: \widetilde{N}^{11}\rightarrow \mathbb{R}$ by $\widetilde{F}\circ \pi_1=f$. The equality $|\nabla^{\widetilde{N}}\widetilde{F}|_{\widetilde{N}}^2=1-\widetilde{F}^2$ has been explained in the introduction. So we are only left to prove the second equality in (\ref{isop func F1}).

At any point $(Q,\theta)\in \widetilde{N}^{11}$, we see that $\nabla^{\widetilde{N}}\widetilde{F}=(0, -\sin\theta)$.
Take an orthonormal basis on $\widetilde{N}^{11}$ as follows:
\begin{equation*}
  \begin{array}{lll}
e_1=\left(Q \left(\begin{array}{cc}
\frac{1}{\sqrt{\lambda(\theta)}}\oi & 0\\
0&0
\end{array}\right), 0 \right), & e_2=\left(Q \left(\begin{array}{cc}
\frac{1}{\sqrt{\lambda(\theta)}}\oj & 0\\
0&0
\end{array}\right), 0 \right), & e_3=\left(Q \left(\begin{array}{cc}
\frac{1}{\sqrt{\lambda(\theta)}}\ok & 0\\
0&0
\end{array}\right), 0 \right)\\
  e_4= \left(Q \left(\begin{array}{cc}
0 & 1\\
-1&0
\end{array}\right), 0 \right),&e_5=\left(Q \left(\begin{array}{cc}
0 & \oi\\
\oi&0
\end{array}\right), 0 \right)
,&e_6=\left(Q \left(\begin{array}{cc}
0 & \oj\\
\oj&0
\end{array}\right), 0 \right) \\
 e_7=\left(Q \left(\begin{array}{cc}
0 & \ok\\
\ok&0
\end{array}\right), 0 \right), & e_8=\left(Q  \Big(\begin{array}{cc}
0 & 0\\
0& \sqrt{\frac{2}{r_2}}\oi
\end{array}\Big), 0\right),&
  e_9=\left(Q  \Big(\begin{array}{cc}
0 & 0\\
0& \sqrt{\frac{2}{r_2}} \oj
\end{array}\Big), 0\right), \\
  e_{10}=\left(Q \Big(\begin{array}{cc}
0 & 0\\
0& \sqrt{\frac{2}{r_2}} \ok
\end{array}\Big), 0\right),& e_{11}=\big(0, 1\big).&
   \end{array}
\end{equation*}
Recalling that we have chosen the normal direction of the hypersurface $Sp(2)_{\theta}$ as $N = (0, -1)$, we will calculate $\widetilde{\Delta}^{\widetilde{N}}\widetilde{F}$ by virtue of the connection (\ref{connection on N11}).
For $i=1,\cdots,10$, noticing $[e_i, \nabla^{\widetilde{N}}\widetilde{F}]=0$, we calculate $E(e_i, \nabla^{\widetilde{N}}\widetilde{F})$ and
find that
\begin{eqnarray*}
  \nabla_{e_1}\nabla^{\widetilde{N}}\widetilde{F}&=&\left(Q\left(\begin{array}{cc}
                                                                  \frac{\lambda'\sin\theta}{2\lambda\sqrt{\lambda}}\oi & 0 \\
                                                                  0 & 0
                                                                \end{array}
  \right), 0\right),\nabla_{e_2}\nabla^{\widetilde{N}}\widetilde{F}=\left(Q\left(\begin{array}{cc}
                                                                 \frac{\lambda'\sin\theta}{2\lambda\sqrt{\lambda}}\oj & 0 \\
                                                                  0 & 0
                                                                \end{array}
  \right), 0\right)
  \\
  \nabla_{e_3}\nabla^{\widetilde{N}}\widetilde{F}&=&\left(Q\left(\begin{array}{cc}
                                                                  \frac{\lambda'\sin\theta}{2\lambda\sqrt{\lambda}}\ok & 0 \\
                                                                  0 & 0
                                                                \end{array}
  \right), 0\right),
  \nabla_{e_4}\nabla^{\widetilde{N}}\widetilde{F}=\cdots=\nabla_{e_{10}}\nabla^{\widetilde{N}}\widetilde{F}=0.
\end{eqnarray*}
Moreover,
\begin{equation*}
  \nabla_{e_{11}}\nabla^{\widetilde{N}}\widetilde{F}=-\nabla_{e_{11}}(\sin\theta e_{11})=(0, -\cos\theta).
\end{equation*}
Therefore,
\begin{eqnarray*}
\widetilde{\Delta}^{\widetilde{N}}\widetilde{F} &=& \sum_{i=1}^{11}\langle \nabla_{e_i}\nabla^{\widetilde{N}}\widetilde{F}, e_i\rangle = -\frac{3\lambda'}{2\lambda}\sin\theta-\cos\theta \\
   &=& -\cos\theta\left( 1+\frac{3}{1+\frac{2}{r_1}\sin^2\theta}\right) = -\widetilde{F}\left( 1+\frac{3}{1+\frac{2}{r_1}(1-\widetilde{F}^2)}\right),
\end{eqnarray*}
which yields (\ref{isop func F1}) and completes the proof of Proposition \ref{isoparametric function F1}.

\hfill $\Box$

\section{Geometry of transnormal system on $N^8$}

In this section, we will first prove the following Theorem.

\noindent
\textbf{Theorem 1.5.}
\emph{(1).The function $F$ defined by $F\circ \pi_2=\widetilde{F}$ is a transnormal function on $N^8$ satisfying $|\nabla^NF|_N^2=1-F^2$. The transnormal hypersurfaces $\Sigma^7_{\theta}$ ($\theta\in (0, \pi)$) are diffeomorphic to the Gromoll-Meyer sphere $\Sigma^7$. Moreover, $\Sigma^7_{\frac{\pi}{2}}$ is totally geodesic in $N^8$. The focal submanifolds $F^{-1}(\pm 1)$ are diffeomorphic to $S^4$, and totally geodesic in $N^8$.}

\emph{(2).For each $\theta\in (0,\pi)$, the induced metric of $\Sigma^7_{\theta}$ in $N^8$ has positive Ricci curvature and quasi-positive sectional curvature simultaneously.}

\vspace{5mm}
\noindent
\textbf{Proof of (1):}

Recall the $S^3$ action on $\widetilde{N}^{11}$ defined in (\ref{S3 action on N11}):
\begin{eqnarray*}
  \phi_2: \,S^3 \times \widetilde{N}^{11}\,\, &\longrightarrow& \quad \widetilde{N}^{11} \\
  q, [(Q, (t_1, t_2))] &\longmapsto& [(\Big(\begin{array}{cc} q &\\ & q\end{array}\Big) Q, (t_1, t_2\overline{q}))] ,\nonumber
\end{eqnarray*}
which is free and isometric.

Define $F: N^8\rightarrow \mathbb{R}$ by $F\circ \pi_2=\widetilde{F}$.
At any point $[(Q, (t_1, t_2))]\in \widetilde{N}^{11}$, we see that the vertical distribution is
\begin{equation*}
  \mathcal{V}_{[(Q, (t_1, t_2))]}=\Big\{\left( \Big(\begin{array}{cc}
                                  u &  \\
                                   & u
                                \end{array}
  \Big)Q, (0, -t_2u)\right)~\big|~u\in \mathbb{H}, \mathrm{Re}(u)=0\Big\}.
\end{equation*}
As the projection of $\nabla^Mf$ on $\widetilde{N}^{11}$, $\nabla^{\widetilde{N}} \widetilde{F}$ is orthogonal to $\mathcal{V}_{[(Q, (t_1, t_2))]}$,
thus a horizontal direction of the Riemannian submersion $\pi_2$ and furthermore
$$|\nabla^NF|_N^2=|\nabla^{\widetilde{N}} \widetilde{F}|_{\widetilde{N}}^2=1-F^2.$$
That is to say, $F$ is a transnormal function on $N^8$. Notice that each transnormal hypersurface of this transnormal system is
a projection of the hypersurface $Sp(2)_{\theta}$ in the isoparametric foliation on $\widetilde{N}^{11}$. Since $Sp(2)_{\frac{\pi}{2}}$ is
totally geodesic in $\widetilde{N}^{11}$, its projection $\Sigma^7_{\frac{\pi}{2}}$ is also totally geodesic.

As for the focal varieties $F^{-1}(\pm 1)$ of $F$, observe that it is the projection of the focal submanifold $S^7$ in $\widetilde{N}^{11}$, and is actually the quotient manifold of the following $S^3\times S^3$-action on $Sp(2)$:
\begin{eqnarray*}
  (S^3 \times S^3)\times Sp(2) \,\, &\longrightarrow& \quad Sp(2) \\
  (p, q, Q) &\longmapsto& \left(\begin{array}{cc}
          p& \\
          & p
        \end{array}\right)Q\left(\begin{array}{cc}
          \overline{q}& \\
          & 1
        \end{array}\right) .\nonumber
\end{eqnarray*}
It is showed in \cite{GM74} that the quotient manifold $Sp(2)/S^3\times S^3$ is diffeomorphic to $S^4$. Moreover, $S^4$ is totally geodesic in $N^8$,  since
$S^7$ is totally geodesic in $\widetilde{N}^{11}$ and $Sp(2)$ is totally geodesic in $M^{14}$ (See Proposition 7.1 in \cite{TT95}).
\hfill $\Box$.

Now we are only left to prove the second part of Theorem \ref{F transnormal}.

\vspace{5mm}
\noindent
\textbf{Proof of (2):}

Fixing $\theta\in (0, \pi)$, for any orthonormal left invariant vector fields $\widetilde{X}=Q\xi\in T_ASp(2)_{\theta}$ with $A=(Q, \theta)\in Sp(2)_{\theta}\subset \widetilde{N}^{11}$, $\xi=\Big(\begin{array}{cc}
                                                        x & y \\
                                                        -\overline{y} & z
                                                      \end{array}
\Big)\in \mathfrak{sp}(2)$,
we find its squared length by the induced metric (\ref{induced metric on N11}) as
\begin{equation*}
  |\widetilde{X}|^2=\frac{\sin^2\theta}{1+\frac{2}{r_1}\sin^2\theta}|x|^2+|y|^2+\frac{r_2}{2}|z|^2:=\frac{\overline{r}_1}{2}|x|^2+|y|^2+\frac{r_2}{2}|z|^2.
\end{equation*}
Notice that $\overline{r}_1=r_1\cdot\frac{2\sin^2\theta}{r_1+2\sin^2\theta}<r_1$, which implies directly $\overline{r}_1+r_2<2$. Thus $Sp(2)_{\theta}$
has non-negative sectional curvatures by part (1) of Theorem \ref{1.1}.

For any orthonormal vectors $X_1, X_2$ on the transnormal hypersurface $\Sigma^7_{\theta}$,
we denote their horizontal lifts in $Sp(2)_{\theta}$ by $\widetilde{X_i}=Q\xi_i\in T_ASp(2)_{\theta}$ $(i=1, 2)$ with $\xi_i=\Big(\begin{array}{cc}
                                                        x_i & y_i \\
                                                       -\overline{y}_i & z_i
                                                      \end{array}
\Big)\in \mathfrak{sp}(2)$. By the Gray-O'Neill formula,
\begin{equation*}
  K(X_1, X_2)\circ \pi_2=\widetilde{K}(\widetilde{X}_1, \widetilde{X}_2)+\frac{3}{4}\|[\widetilde{X}_1, \widetilde{X}_2]\|^2,
\end{equation*}
we get immediately that $ K(X_1, X_2)\geq \widetilde{K}(\widetilde{X}_1, \widetilde{X}_2)\geq 0$.

Next, we will explain that $\widetilde{K}(\widetilde{X}_1, \widetilde{X}_2)$ can not vanish at any point.
From the curvature formula (\ref{sectional cuervature of xi1 xi2}), it follows that
\begin{eqnarray}\label{K on Sp2theta}
  \widetilde{K}(\widetilde{X_1}, \widetilde{X_2}) &=& \widetilde{K}(Q\xi_1, Q\xi_2)=\langle R(Q\xi_1, Q\xi_2)Q\xi_1, Q\xi_2\rangle \nonumber\\
   &=& \frac{1}{4}|\overline{r}_1\gamma_1+r_2\gamma_2|^2+\frac{\overline{r}_1}{8}|\beta_1+(3-2\overline{r}_1)\alpha_1|^2
   +\frac{r_2}{8}|\beta_2+(3-2r_2)\alpha_2|^2\nonumber\\
  && + \frac{1}{2}\big((1-\overline{r}_1)^3+(1-r_2)^3\big)|\alpha_1|^2.
\end{eqnarray}

Therefore,
\begin{equation*}
  \widetilde{K}(\widetilde{X_1}, \widetilde{X_2})=0\quad \Longleftrightarrow \left\{\begin{array}{c}
                                                                                      \alpha_1=\alpha_2=\beta_1=\beta_2=0 \\
                                                                                      \overline{r}_1\gamma_1+r_2\gamma_2=0.
                                                                                    \end{array}\right.
\end{equation*}

Observe that
\begin{equation}\label{parallel}
\alpha_1=\alpha_2=\beta_1=\beta_2=0\,\, \Longrightarrow x_1\sslash x_2,\, y_1\sslash y_2,\, z_1 \sslash z_2.
\end{equation}
Since $\Sigma^7_{\theta}$ is $7$-dimensional, we can find $\xi_0,\cdots,\xi_6$ with $\langle \xi_i, \xi_j\rangle=\delta_{ij}$ and $\widetilde{X_i}=Q\xi_i$ $(i, j=0,\cdots 6)$.

Without loss of generality, we suppose $Ric(\pi_{2{\ast}}\widetilde{X_0})=0$, which implies that $\widetilde{K}(\widetilde{X_0}, \widetilde{X_i})=0$ for any $i=1,\cdots, 6$. Then we discuss case by case:

\noindent
\textbf{Case 1:} $x_0\neq0$, $y_0\neq 0$ and $z_0\neq 0$. By virtue of the observation (\ref{parallel}), there exists $x, y, z\in \mathbb{H}$ with
$|x|=|y|=|z|=1$, such that $x_i=\lambda_ix$, $y_i=\tau_iy$ and $z_i=\mu_iz$ for $i=1,\cdots,6$, where $\lambda_i, \tau_i, \mu_i\in\mathbb{R}$.
Here $(\lambda_i, \tau_i, \mu_i)\in \mathbb{R}^{3}$, but $\dim Span\{\xi_1,\cdots,\xi_6\}=6$, which is a contradiction.

\noindent
\textbf{Case 2:} $x_0\neq0$, $y_0\neq 0$ and $z_0= 0$, i.e., there exists $x, y\in \mathbb{H}$ with $|x|=|y|=1$ such that $\xi_0=(\lambda_0 x, \tau_0 y, 0)$ with $\lambda_0,\tau_0\in\mathbb{R}$, $\lambda_0\tau_0\neq 0$. Then the observation (\ref{parallel}) leads to $\xi_i=(\lambda_ix,\tau_i y, z_i)$ $(i=1,\cdots, 6)$. However, $(\lambda_i, \tau_i, z_i)\in \mathbb{R}^{5}$, which is a contradiction.

\noindent
\textbf{Case 3:} $x_0\neq0$, $y_0= 0$ and $z_0\neq 0$, i.e., there exists $x, z\in \mathbb{H}$ with $|x|=|z|=1$ such that $\xi_0=(\lambda_0 x, 0, \mu_0z)$ with $\lambda_0,\mu_0\in\mathbb{R}$, $\lambda_0\mu_0\neq 0$. Here $(\lambda_i, y_i, \mu_i)\in \mathbb{R}^{6}$. By $0=\langle \xi_0, \xi_i\rangle=\frac{\overline{r}_1}{2}\lambda_0\lambda_i+\frac{r_2}{2}\mu_0\mu_i$, $i=1,\cdots, 6$, we see $\lambda_i$ and $\mu_i$ are in proportion, which is a
contradiction.

\noindent
\textbf{Case 4:} $x_0=0$, $y_0\neq0$ and $z_0\neq 0$,  similar as in case 2.

\noindent
\textbf{Case 5:} $x_0=0$, $y_0=0$ and $z_0\neq 0$, i.e., there exists $z\in \mathbb{H}$ with $|z|=1$ such that $\xi_0=(0, 0, \mu_0z)$ with $\mu_0\in\mathbb{R}$, $\mu_0\neq 0$. Then the observation (\ref{parallel}) leads to $\xi_i=(x_i, y_i, \mu_iz)$ $(i=1,\cdots, 6)$. Then $\langle \xi_0, \xi_i\rangle=\frac{r_2}{2}\langle \mu_iz, \mu_0z\rangle=\frac{r_2}{2}\mu_i\mu_0$ implies that $\mu_i=0$, i.e., $\xi_i=(x_i, y_i, 0)$. In this case,
$\gamma_1=0$, $\gamma_2=-\mu_0y_iz$. Thus the condition $\overline{r}_1\gamma_1+r_2\gamma_2=0$ implies $y_i=0$, i.e., $\xi_i=(x_i, 0, 0)$, which is a contradiction.

\noindent
\textbf{Case 6:} $x_0=0$, $y_0\neq0$ and $z_0= 0$, i.e., there exists $y\in \mathbb{H}$ with $|y|=1$ such that $\xi_0=(0, \tau_0y, 0)$ with $\tau_0\in\mathbb{R}$, $\tau_0\neq 0$. Then the observation (\ref{parallel}) leads to $\xi_i=(x_i, \tau_iy, z_i)$ $(i=1,\cdots, 6)$.
Then $\langle \xi_0, \xi_i\rangle=\langle \tau_0y, \tau_iy\rangle=\tau_0\tau_i$ implies that $\tau_i=0$, i.e., $\xi_i=(x_i, 0, z_i)$.
In this case, $\gamma_1=-\tau_0x_iy$, $\gamma_2=\tau_0yz_i$. Thus the condition $\overline{r_1}\gamma_1+r_2\gamma_2=0$ implies $z_i=\frac{\overline{r_1}}{r_2}\overline{y}x_iy$, i.e., $\xi_i=(x_i, 0, \frac{\overline{r_1}}{r_2}\overline{y}x_iy)$, which is a contradiction.

\noindent
\textbf{Case 7:} $x_0\neq 0$, $y_0=0$ and $z_0= 0$, similar as in case 5.

Consequently, $\widetilde{K}(\widetilde{X_0}, \widetilde{X_i})=0$ $(i=1,\cdots, 6)$ cannot happen simultaneously. In other words, the Ricci curvature is positive at any point of $\Sigma^7_{\theta}$.

At last, we show that the sectional curvature of $\Sigma^7_{\theta}$ at $\pi_2(Id, \theta)$ is positive for all planes.
According to the action
\begin{eqnarray*}
  S^3\times Sp(2)_{\theta}\,\,\, &\longrightarrow& \qquad Sp(2)_{\theta} \\
  p,\quad (Q, \theta) \quad&\longmapsto& (\left(\begin{array}{cc}
          p& \\
          & p
        \end{array}\right)Q\left(\begin{array}{cc}
          \overline{p}& \\
          & 1
        \end{array}\right),\theta),
\end{eqnarray*}
$\pi_2^{-1}([Id])=Id\in Sp(2)_{\theta}$. Moreover, we have
\begin{eqnarray*}
  \mathcal{V}_{Id}&=&\Big\{ \left(\begin{array}{cc}
                                  u &  \\
                                   & u
                                \end{array}
  \right)Id-Id\left(\begin{array}{cc}
                                  u &  \\
                                   & 0
                                \end{array}
  \right)~\big|~u\in \mathbb{H}, \mathrm{Re}(u)=0\Big\}\\
  \mathcal{H}_{Id}&=&\Big\{Q\left(\begin{array}{cc}
                                  x & y \\
                                  -\overline{y} & 0
                                \end{array}
  \right)~\big|~x, y\in \mathbb{H}, \mathrm{Re}(x)=0\Big\}.
\end{eqnarray*}

For $\xi_i=\Big(\begin{array}{cc}
                                  x_i & y_i \\
                                  -\overline{y}_i & 0
                                \end{array}
  \Big)\in \mathcal{H}_{Id}$, $i=1, 2$, by virtue of the sectional curvature formula (\ref{K on Sp2theta}),
$\widetilde{K}(\xi_1, \xi_2)=0$ implies that $\alpha_1=\alpha_2=\beta_1=\beta_2=0$, and $\gamma_1=0$ since $\gamma_2=0$ in this case. Thus $x_1\sslash x_2, y_1\sslash y_2$, and furtherer $\xi_1\sslash \xi_2$, which proves our assertion.

\hfill $\Box$

\begin{ack}
The authors would like to thank Jianquan Ge for helpful discussions.
\end{ack}

\end{document}